\documentclass[11pt]{amsart}

\usepackage[english]{babel}
\usepackage{bm}
\usepackage{fullpage}
\usepackage{amssymb}
\usepackage{amsmath}
\usepackage{latexsym}
\usepackage{graphicx}

\newtheorem{lemma}{Lemma}

\newtheorem{thm}{Theorem}
\newtheorem{corollary}{Corollary}
\newtheorem{conjecture}{Conjecture}

\title{The chromatic number of almost stable Kneser hypergraphs}
\author{Fr\'ed\'eric Meunier}
\address{Universit\'e Paris Est, LVMT, ENPC, 6-8 avenue Blaise Pascal, Cit\'e Descartes
Champs-sur-Marne, 77455 Marne-la-Vall\'ee cedex 2, France.}
\email{frederic.meunier@enpc.fr}

\begin{document}

\maketitle

\begin{abstract}
Let $V(n,k,s)$ be the set of $k$-subsets $S$ of $[n]$ such that for all $i,j\in S$, we have $|i-j|\geq s$
We define almost $s$-stable Kneser hypergraph $KG^r{{[n]}\choose k}_{ s\mbox{\tiny{\textup{-stab}}}}^{\displaystyle\sim}$ to be the $r$-uniform hypergraph whose vertex set is $V(n,k,s)$ and whose edges are the $r$-uples of disjoint elements of $V(n,k,s)$.

With the help of a $Z_p$-Tucker lemma, we prove that, for $p$ prime and  for any $n\geq kp$, the chromatic number of almost $2$-stable Kneser hypergraphs $KG^p {{[n]}\choose k}_{2\mbox{\tiny{\textup{-stab}}}}^{\displaystyle\sim}$ is equal to the chromatic number of the usual Kneser hypergraphs $KG^p{{[n]}\choose k}$, namely that it is equal to $\left\lceil\frac{n-(k-1)p}{p-1}\right\rceil.$

Defining $\mu(r)$ to be the number of prime divisors of $r$, counted with multiplicities, this result implies that 
the chromatic number of almost $2^{\mu(r)}$-stable Kneser hypergraphs $KG^r{{[n]}\choose k}_{ 2^{\mu(r)}\mbox{\tiny{\textup{-stab}}}}^{\displaystyle\sim}$ is equal to the chromatic number of the usual Kneser hypergraphs $KG^r{{[n]}\choose k}$ for any $n\geq kr$, namely that it is equal to $\left\lceil\frac{n-(k-1)r}{r-1}\right\rceil.$
\end{abstract}

\section{Introduction and main results}

Let $[a]$ denote the set $\{1,\ldots,a\}$.
The Kneser graph $KG^2{{[n]}\choose k}$ for integers $n\geq 2k$ is defined as follows: its vertex set is the set of $k$-subsets of $[n]$ and two vertices are connected by an edge if they have an empty intersection.

Kneser conjectured \cite{Kn55} in 1955 that its chromatic number $\chi\left(KG^2{{[n]}\choose k}\right)$ is equal to $n-2k+2$. It was proved to be true by Lov\'asz in 1979 in a famous paper \cite{Lo79}, which is the first and one of the most spectacular application of algebraic topology in combinatorics.

Soon after this result, Schrijver \cite{Sc78} proved that the chromatic number remains the same when we consider the subgraph $KG^2{{[n]}\choose k}_{2\mbox{\tiny{\textup{-stab}}}}$ of $KG^2{{[n]}\choose k}$ obtained by restricting the vertex set to the $k$-subsets that are {\em $2$-stable}, that is, that do not contain two consecutive elements of $[n]$ (where $1$ and $n$ are considered to be also consecutive).

Let us recall that an {\em hypergraph} $\mathcal{H}$ is a set family $\mathcal{H}\subseteq2^V$, with {\em vertex set} $V$. An hypergraph is said to be {\em $r$-uniform} if all its {\em edges} $S\in\mathcal{H}$ have the same cardinality $r$. A {\em proper coloring with $t$ colors} of $\mathcal{H}$ is a map $c:V\rightarrow[t]$ such that there is no monochromatic edge, that is such that in each edge there are two vertices $i$ and $j$ with $c(i)\neq c(j)$. The smallest number $t$ such that there exists such a proper coloring is called {\em the chromatic number} of $\mathcal{H}$ and denoted by $\chi(\mathcal{H})$.
 
In 1986, solving a conjecture of Erd\H{o}s \cite{Er76}, Alon, Frankl and Lov\'asz \cite{AlFrLo86} found the chromatic number of {\em Kneser hypergraphs}. The Kneser hypergraph $KG^r{{[n]}\choose k}$ is a $r$-uniform hypergraph which has the $k$-subsets of $[n]$ as vertex set and whose edges are formed by the $r$-uple of disjoint $k$-subsets of $[n]$. Let $n,k,r,t$ be positive integers such that $n\geq(t-1)(r-1)+rk$. Then $\chi\left(KG^r{{[n]}\choose k}\right)>t$. Combined with a lemma by Erd\H{o}s giving an explicit proper coloring, it implies that $\chi\left(KG^r{{[n]}\choose k}\right)=\left\lceil\frac{n-(k-1)r}{r-1}\right\rceil$. The proof found by Alon, Frankl and Lov\'asz used tools from algebraic topology.

In 2001, Ziegler gave a combinatorial proof of this theorem \cite{Zi02}, which makes no use of homology, simplicial approximation,... He was inspired by a combinatorial proof of the Lov\'asz theorem found by Matou\v{s}ek \cite{Ma04}. 
A subset $S\subseteq [n]$ is {\em $s$-stable} if any two of its elements are at least ``at distance $s$ apart'' on the
$n$-cycle, that is, if $s\leq |i - j| \leq n-s$ for distinct $i, j\in S$. Define then $KG^r{{[n]}\choose k}_{\mbox{\tiny{$s$-stab}}}$ as the hypergraph obtained by restricting the vertex set of $KG^r{{[n]}\choose k}$ to the $s$-stable $k$-subsets.
At the end of his paper, Ziegler made the supposition that the chromatic number of $KG^r{{[n]}\choose k}_{\mbox{\tiny{$r$-stab}}}$ is equal to the chromatic number of $KG^r{{[n]}\choose k}$ for any $n\geq kr$.
This supposition generalizes both Schrijver's theorem and the Alon-Frankl-Lov\'asz theorem. Alon, Drewnowski and \L ucsak make this supposition an explicit conjecture in \cite{AlDrLu09}.

\begin{conjecture}\label{conj:zie}
Let $n,k,r$ be non-negative integers such that $n\geq rk$. Then $$\chi\left(KG^r{{[n]}\choose k}_{\mbox{\tiny{$r$-stab}}}\right)=\left\lceil\frac{n-(k-1)r}{r-1}\right\rceil.$$
\end{conjecture}

We prove a weaker form of this statement, but which strengthes the Alon-Frankl-Lov\'asz theorem. 
Let $V(n,k,s)$ be the set of $k$-subsets $S$ of $[n]$ such that for all $i,j\in S$, we have $|i-j|\geq s$
We define the almost $s$-stable Kneser hypergraphs $KG^r{{[n]}\choose k}_{ s\mbox{\tiny{\textup{-stab}}}}^{\displaystyle\sim}$ to be the $r$-uniform hypergraph whose vertex set is $V(n,k,s)$ and whose edges are the $r$-uples of disjoint elements of $V(n,k,s)$.

\begin{thm}\label{main}
Let $p$ be a prime number and $n,k$ be non negative integers such that $n\geq pk$. We have $$\chi\left(KG^p{{[n]}\choose k}_{ 2\mbox{\tiny{\textup{-stab}}}}^{\displaystyle\sim}\right)\geq \left\lceil\frac{n-(k-1)p}{p-1}\right\rceil.$$
\end{thm}

Combined with the lemma by Erd\H{o}s, we get that $$\chi\left(KG^p{{[n]}\choose k}_{ 2\mbox{\tiny{\textup{-stab}}}}^{\displaystyle\sim}\right)= \left\lceil\frac{n-(k-1)p}{p-1}\right\rceil.$$

Moreover, we will see that it is then possible to derive the following corollary. Denote by $\mu(r)$ the number of prime divisors of $r$ counted with multiplicities. For instance, $\mu(6)=2$ and $\mu(12)=3$. We have

\begin{corollary}\label{cor:gen}
Let $n,k,r$ be non-negative integers such that $n\geq rk$. We have $$KG^r{{[n]}\choose k}_{ 2^{\mu(r)}\mbox{\tiny{\textup{-stab}}}}^{\displaystyle\sim}=\left\lceil\frac{n-(k-1)r}{r-1}\right\rceil.$$
\end{corollary}

\section{Notations and tools}

$Z_p=\{\omega,\omega^2,\ldots,\omega^p\}$ is the cyclic group of order $p$, with generator $\omega$. 

We write $\sigma^{n-1}$ for the $(n-1)$-dimensional simplex with vertex set $[n]$ and by $\sigma_{k-1}^{n-1}$ the $(k-1)$-skeleton of this simplex, that is the set of faces of $\sigma^{n-1}$ having $k$ or less vertices.

If $A$ and $B$ are two sets, we write $A\uplus B$ for the set $(A\times\{1\})\cup(B\times\{2\})$. For two simplicial complexes, $\mathsf{K}$ and $\mathsf{L}$, with vertex sets $V(\mathsf{K})$ and $V(\mathsf{L})$, we denote by $\mathsf{K}*\mathsf{L}$ the {\em join} of these two complexes, which is the simplicial complex having $V(\mathsf{K})\uplus V(\mathsf{L})$ as vertex set and $$\{F\uplus G:\,F\in\mathsf{K},G\in\mathsf{L}\}$$ as set of faces.
We define also $\mathsf{K}^{*n}$ to be the join of $n$ disjoint copies of $\mathsf{K}$.

Let $X=(x_1,\ldots,x_n)\in(Z_p\cup\{0\})^{n}$. We denote by $\mbox{alt}(X)$ the size of the longest alternating subsequence of non-zero terms in $X$. A sequence $(j_1,j_2,\ldots,j_m)$ of elements of $Z_p$ is said to be {\em alternating} if any two consecutive terms are different. For instance (assume $p=5$) $\mbox{alt}(\omega^2,\omega^3,0,\omega^3,\omega^5,0,0,\omega^2)=4$ and $\mbox{alt}(\omega^1,\omega^4,\omega^4,\omega^4,0,0,\omega^4)=2$.

Any element element $X=(x_1,\ldots,x_n)\in(Z_p\cup\{0\})^{n}$ can alternatively and without further mention be denoted by a $p$-uple $(X_1,\ldots,X_p)$ where $X_j:=\{i\in[n]:\,x_i=\omega^j\}$. Note that the $X_j$ are then necessarily disjoint. For two elements $X,Y\in(Z_p\cup\{0\})^{n}$, we denote by $X\subseteq Y$ the fact that for all $j\in[p]$ we have $X_j\subseteq Y_j$. When $X\subseteq Y$, note that the sequence of non-zero terms in $(x_1,\ldots,x_n)$ is a subsequence of $(y_1,\ldots,y_n)$.

The proof of Theorem  \ref{main} makes use of a variant of the $Z_p$-Tucker lemma by Ziegler \cite{Zi02}.
\begin{lemma}[$Z_p$-Tucker lemma]\label{lem:zptucker} Let $p$ be a prime, $n,m\geq 1$, $\alpha\leq m$ and let
$$\begin{array}{lccc}\lambda:
& (Z_p\cup\{0\})^{n}\setminus\{(0,\ldots,0)\} &
\longrightarrow & Z_p\times [m] \\
& X & \longmapsto & (\lambda_1(X),\lambda_2(X))
\end{array}$$ be a
$Z_p$-equivariant map satisfying the following properties:
\begin{itemize}
\item for all $X^{(1)}\subseteq X^{(2)}\in(Z_p\cup\{0\})^{n}\setminus\{(0,\ldots,0)\}$, if $\lambda_2(X^{(1)})=\lambda_2(X^{(2)})\leq\alpha$, then $\lambda_1(X^{(1)})=\lambda_1(X^{(2)})$;
\item
for all
$X^{(1)}\subseteq X^{(2)}\subseteq \ldots \subseteq X^{(p)}\in(Z_p\cup\{0\})^{n}\setminus\{(0,\ldots,0)\}$, if $\lambda_2(X^{(1)})=\lambda_2(X^{(2)})=\ldots=\lambda_2(X^{(p)})\geq\alpha+1$, then the $\lambda_1(X^{(i)})$ are not pairwise distinct for $i=1,\ldots,p$. 
\end{itemize}

Then $\alpha+(m-\alpha)(p-1)\geq n$.
\end{lemma}

We can alternatively say that $X\mapsto\lambda(X)=(\lambda_1(X),\lambda_2(X))$ is a $Z_p$-equivariant simplicial map from $\mbox{sd}\left(Z_p^{*n}\right)$ to $\left(Z_p^{*{\alpha}}\right)*\left((\sigma_{p-2}^{p-1})^{*(m-\alpha)}\right)$, where $\mbox{sd}(\mathsf{K})$ denotes the fist barycentric subdivision of a simplicial complex $\mathsf{K}$.

\begin{proof}[Proof of the $Z_p$-Tucker lemma]
According to Dold's theorem \cite{Do83,Ma03}, if such a map $\lambda$ exists, the dimension of $\left(Z_p^{*{\alpha}}\right)*\left((\sigma_{p-2}^{p-1})^{*(m-\alpha)}\right)$ is strictly larger than the connectivity of $Z_p^{*n}$, that is $\alpha+(m-\alpha)(p-1)-1>n-2$.
\end{proof}

It is also possible to give a purely combinatorial proof of this lemma through the generalized Ky Fan theorem from \cite{HaSaScZi08}.

\section{Proof of the main results}

\begin{proof}[Proof of Theorem \ref{main}]
We follow the scheme used by Ziegler in \cite{Zi02}. We endow $2^{[n]}$ with an arbitrary linear order $\preceq$.

Assume that $KG^p{{[n]}\choose k}_{2\mbox{\tiny{\textup{-stab}}}}^{\displaystyle\sim}$ is properly colored with $C$ colors $\{1,\ldots, C\}$. For $S\in V(n,k,2)$, we denote by $c(S)$ its color. Let $\alpha=p(k-1)$ and $m=p(k-1)+C$.

Let $X=(x_1,\ldots,x_n)\in(Z_p\cup\{0\})^{n}\setminus\{(0,\ldots,0)\}$. 
We can write alternatively $X=(X_1,\ldots,X_p)$.

\begin{itemize}
\item if $\mbox{alt}(X)\leq p(k-1)$, let $j$ be the index of the $X_j$ containing the smallest integer ($\omega^j$ is then the first non-zero term in $(x_1,\ldots,x_n)$), and define $$\lambda(X):=(j,\mbox{alt}(X)).$$
\item if $\mbox{alt}(X)\geq p(k-1)+1$: in the longest alternating subsequence of non-zero terms of $X$, at least one of the elements of $Z_p$ appears at least $k$ times; hence, in at least one of the $X_j$ there is an element $S$ of $V(n,k,2)$; choose the smallest such $S$ (according to $\preceq$).
Let $j$ be such that $S\subseteq X_j$ and define $$\lambda(X):=(j,c(S)+p(k-1)).$$
\end{itemize}

$\lambda$ is $Z_p$-equivariant map from $(Z_p\cup\{0\})^{n}\setminus\{(0,\ldots,0)\}$ to $Z_p\times [m]$. 

Let $X^{(1)}\subseteq X^{(2)}\in(Z_p\cup\{0\})^{n}\setminus\{(0,\ldots,0)\}$. If $\lambda_2(X^{(1)})=\lambda_2(X^{(2)})\leq\alpha$, then the longest alternating subsequences of non-zero terms of $X^{(1)}$ and $X^{(2)}$ have same size. Clearly, the first non-zero terms of $X^{(1)}$ and $X^{(2)}$ are equal.

Let
$X^{(1)}\subseteq X^{(2)}\subseteq \ldots \subseteq X^{(p)}\in(Z_p\cup\{0\})^{n}\setminus\{(0,\ldots,0)\}$. If $\lambda_2(X^{(1)})=\lambda_2(X^{(2)})=\ldots=\lambda_2(X^{(p)})\geq\alpha+1$, then for each $i\in[p]$ there is $S_i\in V(n,k,2)$ and $j_i\in[p]$ such that we have $S_i\subseteq X_{j_i}^{(i)}$ and $\lambda_2(X^{(i)})=c(S_i)$. If all $\lambda_1(X^{(i)})$ would be distinct, then it would mean that all $j_i$ would be distinct, which implies that the $S_i$ would be disjoint but colored with the same color, which is impossible since $c$ is a proper coloring.

We can thus apply the $Z_p$-Tucker lemma (Lemma \ref{lem:zptucker}) and conclude that $n\leq p(k-1)+C(p-1)$, that is 
$$C\geq\left\lceil\frac{n-(k-1)p}{p-1}\right\rceil.$$
\end{proof}

To prove Corollary \ref{cor:gen}, we prove the following lemma, both statement and proof of which are inspired by Lemma 3.3 of \cite{AlDrLu09}.

\begin{lemma}\label{lem:gen}
Let $r_1,r_2,s_1,s_2$ be non-negative integers $\geq 1$, and define $r=r_1r_2$ and $s=s_1s_2$.

Assume that for $i=1,2$ we have $\chi\left(KG^{r_i}{{[n]}\choose k}_{ s_i\mbox{\tiny{\textup{-stab}}}}^{\displaystyle\sim}\right)=\left\lceil\frac{n-(k-1)r_i}{r_i-1}\right\rceil$ for all integers $n$ and $k$ such that $n\geq r_ik$.

Then we have $\chi\left(KG^{r}{{[n]}\choose k}_{ s\mbox{\tiny{\textup{-stab}}}}^{\displaystyle\sim}\right)=\left\lceil\frac{n-(k-1)r}{r-1}\right\rceil$ for all integers $n$ and $k$ such that $n\geq rk$.
\end{lemma}

\begin{proof}
Let $n\geq (t-1)(r-1)+rk$. We have to prove that $\chi\left(KG^{r}{{[n]}\choose k}_{ s\mbox{\tiny{\textup{-stab}}}}^{\displaystyle\sim}\right)>t$. For a contradiction,
assume that $KG^r{{[n]}\choose k}_{s\mbox{\tiny{-stab}}}$ is properly colored with $C\leq t$ colors. For $S\in V(n,k,p)$, we denote by $c(S)$ its color. We wish to prove that there are $S_1,\ldots,S_r$ disjoint elements of $V(n,k,s)$ with $c(S_1)=\ldots=c(S_r)$.

Take $A\in V(n,n_1,s_1)$, where $n_1:=r_1k+(t-1)(r_1-1)$. Denote $a_1<\ldots<a_{n_1}$ the elements of $A$ and define $h:V(n_1,k,s_2)\rightarrow[t]$ as follows: let $B\in V(n_1,k,s_2)$; the $k$-subset $S=\{a_i:\,i\in B\}\subseteq[n]$ is an element of $V(n,k,s)$, and gets as such a color $c(S)$; define $h(B)$ to be this $c(S)$.
Since $n_1=r_1k+(t-1)(r_1-1)$, there are $B_1,\ldots,B_{r_1}$ disjoint elements of $V(n_1,k,s_2)$ having the same color by $h$. Define $\tilde{h}(A)$ to be this common color.

Make the same definition for all $A\in V(n,n_1,s_1)$. The map $\tilde{h}$ is a coloring of $KG^{r_2}{{[n]}\choose{n_1}}_{s_1\mbox{\tiny{\textup{-stab}}}}^{\displaystyle\sim}$ with $t$ colors.
Now, note that $$(t-1)(r-1)+rk=(t-1)(r_1r_2-r_2+r_2-1)+r_1r_2k=(t-1)(r_2-1)+r_2((t-1)(r_1-1)+r_1k)$$ and thus that $n\geq(t-1)(r_2-1)+r_2n_1$. Hence, there are $A_1,\ldots,A_{r_2}$ disjoint elements of $V(n,n_1,s_1)$ with the same color. Each of the $A_i$ gets its color from $r_1$ disjoint elements of $V(n,k,s)$, whence there are $r_1r_2$ disjoint elements of $V(n,k,s)$ having the same color by the map $c$.
\end{proof}

\begin{proof}[Proof of Corollary \ref{cor:gen}]
Direct consequence of Theorem \ref{main} and Lemma \ref{lem:gen}.
\end{proof}

\section{Short combinatorial proof of Schrijver's theorem}

Recall that Schrijver's theorem is 
\begin{thm} Let $n\geq 2k$.
$\chi\left(KG{{[n]}\choose k}_{2\tiny{\textup{-stab}}}\right)=n-2k+2.$
\end{thm}

When specialized for $p=2$, Theorem \ref{main} does not imply Schrijver's theorem since the vertex set is allowed to contain subsets with $1$ and $n$ together. Anyway, by a slight modification of the proof, we can get a short combinatorial proof of Schrijver's theorem. Alternative proofs of this kind -- but not that short -- have been proposed in \cite{Me08,Zi02}

For a positive integer
$n$, we write $\{+,-,0\}^n$ for the set of all {\em signed subsets}
of $[n]$, that is, the family of all pairs $(X^+,X^-)$ of disjoint
subsets of $[n]$. Indeed, for $X\in\{+,-,0\}^n$, we can define
$X^+:=\{i\in[n]: X_i=+\}$ and analogously $X^-$.

We define $X\subseteq Y$ if and only if $X^+\subseteq Y^+$ and $X^-\subseteq Y^-$.

By $\mbox{alt}(X)$ we denote the length of the longest alternating subsequence of non-zero signs in $X$.
For instance: $\mbox{alt}(+0--+0-)=4$, while $\mbox{alt}(--++-+0+-)=5$.

The proof makes use of the following well-known lemma see \cite{Ma03,Tu46,Zi02} (which is a special case of Lemma \ref{lem:zptucker} for $p=2$).

\begin{lemma}[Tucker's lemma]
Let $\lambda: \{-,0,+\}^n\setminus\{(0,0,\ldots,0)\}\rightarrow\{-1,+1,\ldots,-n,+n\}$ be a map such that $\lambda(-X)=-\lambda(X)$. Then there exist $A,B$ in $\{-,0,+\}^n$ such that $A\subseteq B$ and $\lambda(A)=-\lambda(B)$.
\end{lemma}

\begin{proof}[Proof of Schrijver's theorem]

The inequality $\chi\left(KG^2{{[n]}\choose k}_{2\tiny{\textup{-stab}}}\right)\leq n-2k+2$ is easy to prove (with an
explicit coloring) and well-known. So, to obtain a combinatorial
proof, it is sufficient to prove the reverse inequality.

Let us assume that there is a proper coloring $c$ of $KG^2{{[n]}\choose k}_{2\tiny{\textup{-stab}}}$ with $n-2k+1$ colors. We define the following map $\lambda$ on $\{-,0,+\}^n\setminus\{(0,0,\ldots,0)\}$. 
\begin{itemize}
\item if $\mbox{alt}(X)\leq 2k-1$, we define $\lambda(X)=\pm \mbox{alt}(X)$, where the sign is determined by the first sign of the longest alternating subsequence of $X$ (which is actually the first non zero term of $X$).
\item if $\mbox{alt}(X)\geq 2k$, then $X^+$ and $X^-$ both contain a stable subset of $[n]$ of size $k$. Among all stable subsets of size $k$ included in $X^-$ and $X^+$, select the one having the smallest color. Call it $S$. Then define $\lambda(X)=\pm (c(S)+2k-1)$ where the sign indicates which of $X^-$ or $X^+$ the subset $S$ has been taken from. Note that $c(S)\leq n-2k$.
\end{itemize}

The fact that for any $X\in\{-,0,+\}^n\setminus\{(0,0,\ldots,0)\}$ we have $\lambda(-X)=-\lambda(X)$ is obvious. $\lambda$ takes its values in $\{-1,+1,\ldots,-n,+n\}$.
Now let us take $A$ and $B$ as in Tucker's lemma, with $A\subseteq B$ and $\lambda(A)=-\lambda(B)$. We cannot have $\mbox{alt}(A)\leq 2k-1$ since otherwise we will have a longest alternating in $B$ containg the one of $A$, of same length but with a different sign.
Hence $\mbox{alt}(A)\geq 2k$. Assume w.l.o.g. that $\lambda(A)$ is defined by a stable subset $S_A\subseteq A^-$. Then the stable subset $S_B$ defining $\lambda(B)$ is such that $S_B\subseteq B^+$, which implies that $S_A\cap S_B=\emptyset$. We have moreover $c(S_A)=|\lambda(A)|=|\lambda(B)|=c(S_B)$, but this contradicts the fact that $c$ is proper coloring of $KG^2{{[n]}\choose k}_{2\tiny{\textup{-stab}}}$.
\end{proof}

\section{Concluding remarks}

We have seen that one of the main ingredients is the notion of alternating sequence of elements in $Z_p$. Here, our notion only requires that such an alternating sequence must have $x_i\neq x_{i+1}$.
To prove Conjecture \ref{conj:zie}, we need probably something stronger. For example, a sequence is said to be alternating if any $p$ consecutive terms are all distinct. Anyway, all our attempts to get something through this approach have failed.

Recall that Alon, Drewnowski and \L ucsak \cite{AlDrLu09} proved Conjecture \ref{conj:zie} when $r$ is a power of $2$. With the help of a computer and \texttt{lpsolve}, we check that Conjecture \ref{conj:zie} is moreover true for 
\begin{itemize}
\item $n\leq 9$, $k=2$, $r=3$.
\item $n\leq 12$, $k=3$, $r=3$.
\item $n\leq 14$, $k=4$, $r=3$.
\item $n\leq 13$, $k=2$, $r=5$.
\item $n\leq 16$, $k=3$, $r=5$.
\item $n\leq 21$, $k=4$, $r=5$.
\end{itemize}

\bibliographystyle{amsplain}
\bibliography{Combinatorics}
\end{document}